\documentclass[11pt]{article}
\usepackage{color,latexsym,amsthm,amsmath,amssymb,path,enumitem,url,bbm,subcaption,geometry}

\newcommand{\ignore}[1]{}

\newtheorem{theorem}{Theorem}[section]
\newtheorem{lemma}[theorem]{Lemma}
\newtheorem{corollary}[theorem]{Corollary}

\newcommand{\Proof}[1]
        {
        \noindent
        \emph{Proof #1.}~
        }
\newsavebox{\smallProofsym}                     
\savebox{\smallProofsym}                        %
        {
        \begin{picture}(7,7)                    %
        \put(0,0){\framebox(6,6){}}             %
        \put(0,2){\framebox(4,4){}}             %
        \end{picture}                           %
        }                                       %
\newcommand{\smalleop}[1]
        {
        \mbox{} \hfill #1~~\usebox{\smallProofsym}\!\!\!\!\!\!\
        }

\newcommand{\parag}[1]{\vspace{2mm}

\noindent{\bf #1} }

\usepackage{graphicx,psfrag}
\usepackage[rflt]{floatflt}

\usepackage{dsfont}
\newcommand{\NN}{\ensuremath{\mathbb N}}

\newcommand{\RR}{\ensuremath{\mathbb R}}

\newcommand{\pts}{\mathcal P}

\newcommand{\lines}{\mathcal L}

\newcommand{\GCD}{\mathrm{GCD}}

\def\eps{{\varepsilon}}

\begin{document}
\pagenumbering{arabic}

\title{The constant of point--line incidence constructions}

\author{
Martin Balko\thanks{Faculty of Mathematics and Physics, Charles University, Czech Republic. {\sl balko@kam.mff.cuni.cz}. Supported by the grant no. 21/32817S of the Czech Science Foundation (GA\v{C}R) and by the Center for Foundations of Modern Computer Science (Charles University project UNCE/SCI/004).
This article is part of a project that has received funding from the European Research Council (ERC) under the European Union's Horizon 2020 research and innovation programme (grant agreement No 810115).
}
\and
Adam Sheffer\thanks{Department of Mathematics, Baruch College, City University of New York, NY, USA.
{\sl adamsh@gmail.com}. Supported by NSF award DMS-1802059 and by PSCCUNY award 63580.}
\and
Ruiwen Tang\thanks{Stuyvesant High School. 345 Chambers St, New York, NY 10282.
{\sl ruiwentang@gmail.com}.}}

\date{}

\maketitle

\begin{abstract}
We study a lower bound for the constant of the Szemer\'edi--Trotter theorem. 
In particular, we show that a recent infinite family of point-line configurations satisfies $I(\pts,\lines)\ge (c+o(1)) |\pts|^{2/3}|\lines|^{2/3}$, with $c\approx 1.27$.
Our technique is based on studying a variety of properties of Euler's totient function. 
We also improve the current best constant for Elekes's construction from 1 to about 1.27. 
From an expository perspective, this is the first full analysis of the constant of Erd\H os's construction.
\end{abstract}

\section{Introduction}

In recent years, there has been significant interest in finding exact constants for incidence-related problems.
For example, Yu and Zhao \cite{YZ22} derived the constant of the joints problem, and Bukh and Chao \cite{BC21} derived the constant of the finite field Kakeya problem (building upon previous works \cite{DKSS13,SS08}). 
In this work, we explore the constant of what one may call the fundamental theorem of incidence theory: \emph{The Szemer\'edi--Trotter theorem}.

Throughout this paper we work in $\RR^2$.
Let $\pts$ be a finite set of points and let $\lines$ be a finite set of lines.
An \emph{incidence} is a point--line pair $(p,\ell)\in \pts\times \lines$ such that $p\in \ell$.
The number of incidences in $\pts\times \lines$ is denoted as $I(\pts,\lines)$.

\begin{theorem}[Szemer\'edi--Trotter] \label{th:SzemTrot}
Every set of points $\pts$ and set of lines $\lines$ satisfy that
\[ I(\pts,\lines) = O(|\pts|^{2/3}|\lines|^{2/3}+|\pts|+|\lines|). \]
\end{theorem}

Theorem~\ref{th:SzemTrot} has a large number of applications in a wide-variety of fields, including combinatorics, harmonic analysis, number theory, theoretical computer science, and model theory (for example, see \cite{BB15,Bourgain05,CGS20,Demeter14}).
This theorem has been known to be asymptotically tight since its appearance about 40 years ago. 
However, the exact constant remains unknown.
When $|\lines|=O(|\pts|^{1/2})$, the bound of Theorem~\ref{th:SzemTrot} is dominated by the term $|\pts|$.
When $|\lines|=\Omega(|\pts|^2)$, the bound of Theorem~\ref{th:SzemTrot} is dominated by the term $|\lines|$.
These are considered as less interesting extreme cases.
The main case is when\footnote{The notation $X = \omega(Y)$ means ``$X$ is asymptotically larger than $Y$.'' 
The notation $X = o(Y)$ means ``$X$ is asymptotically smaller than $Y$.'' } $|\lines|=\omega(|\pts|^{1/2})$ and $|\lines|=o(|\pts|^2)$, so the dominating term is $|\pts|^{2/3}|\lines|^{2/3}$.

The current best upper bound, taken from \cite{Ackerman19}, is $I(\pts,\lines)\le 2.44|\pts|^{2/3}|\lines|^{2/3}$ (this improves the previous bound of 2.5 from \cite{PRTT04}).
The status of the lower bound is not as simple. 
During the 20th century, only a single construction with $\Theta(|\pts|^{2/3}|\lines|^{2/3})$ incidences was known, due to Erd\H os (see Section \ref{sec:constructions}).
Pach and T\'oth \cite{PachToth97} analyzed Erd\H os's construction, to obtain $c|\pts|^{2/3}|\lines|^{2/3}$ incidences, where $c=(3/4\pi^2)^{1/3} \approx 0.42$.
Most of the steps of this analysis were omitted in the paper \cite{PachToth97}, making it difficult to verify. 
Twenty years later, Cibulka, Valtr, and the first author of the current work \cite{BCV19} spotted a mistake in the part of the analysis in \cite{PachToth97} that did appear on the paper. 
Fixing this mistake leads to the improved bound $c=3\cdot (3/4\pi^2)^{1/3} \approx 1.27$.
It remained unclear whether additional issues appear in the portion of the analysis that is omitted in \cite{PachToth97}.

In the early 2000s, Elekes \cite{Elekes2001} discovered a simpler point--line configuration that achieves 
$\Theta(|\pts|^{2/3}|\lines|^{2/3})$ incidences (see Section \ref{sec:constructions}).
The best constant previously obtained for this construction was by Apfelbaum \cite{Apfelbaum17}, who showed that $I(\pts,\lines)\ge |\pts|^{2/3}|\lines|^{2/3}$.
In other words, Apfelbaum obtained the constant $c=1$.

We set $n=|\pts|$.
In Erd\H os's construction, $\pts$ is a $\sqrt{n}\times \sqrt{n}$ section of the integer lattice.
In Elekes's construction, $\pts$ is a much taller and thinner section of the integer lattice. 
Recently, Silier and the second author of the current work \cite{SS21} discovered an infinite family of constructions with $\Theta(|\pts|^{2/3}|\lines|^{2/3})$ incidences.
In this family, the point set is an $n^\alpha \times n^{1-\alpha}$ section of the integer lattice, where $1/3 \le \alpha\le 1/2$.
Erd\H os's construction is obtained when $\alpha =1/2$ and Elekes's construction is obtained when $\alpha=1/3$.
For a description of this family, see Section \ref{sec:constructions}.
Previously, no work analyzed the constants of these constructions.

The following theorem is the main result of the current work.

\begin{theorem} \label{th:main}
Consider a construction $\pts,\lines$ from the infinite family with $1/3 \le \alpha \le 1/2$, $|\lines|= o(|\pts|^2)$, and $|\lines|= \omega(|\pts|^{2-3\alpha})$.
Then $I(\pts,\lines)=(c+o(1))|\pts|^{2/3}|\lines|^{2/3}$, where $c=3\cdot (3/4\pi^2)^{1/3} \approx 1.27$.
\end{theorem}

The condition $|\lines|= o(|\pts|^2)$ only means that we are in the main case, where the number of incidences is dominated by the term $|\pts|^{2/3}|\lines|^{2/3}$.
The condition $|\lines|= \omega(|\pts|^{2-3\alpha})$ is more restrictive when $\alpha<1/2$.
For example, when $\alpha=1/3$, this condition asks for the number of lines to be asymptotically larger than the number of points. 

We believe that Theorem~\ref{th:main} is of interest for several reasons:
\begin{itemize}[noitemsep,topsep=1pt]
\item Generalizing the analysis of Erd\H os's construction to the entire family is not trivial and requires multiple new ideas.
\item The constant of Elekes's construction is improved from 1 to about 1.27.
\item From an expository perspective, the proof from \cite{PachToth97} is now fully explained and verified. 
This clarifies some details, such as that an additional $+o(1)$ should appear in the bound.
\end{itemize}

Recently, Guth and Silier \cite{GS21} discovered another infinite family of constructions with $\Theta(|\pts|^{2/3}|\lines|^{2/3})$ incidences.
It may be interesting to analyze the constants of this family.  

In Section \ref{sec:constructions}, we describe the constructions of Erd\H os and Elekes, and the infinite family of constructions.
In Section \ref{sec:totient}, we derive a variety of results related to Euler's totient function, which are required in our analysis.
In Section \ref{sec:MainProof}, we prove Theorem~\ref{th:main}.

\section{The constructions} \label{sec:constructions}

The purpose of this section is to briefly provide intuition for Erd\H os's construction, Elekes's construction, and the infinite family of constructions from \cite{SS21}.
For simplicity and intuition, we only describe the case where $|\pts|=|\lines|$ and with simplified sets of lines.
These simpler sets do lead to $\Theta(|\pts|^{2/3}|\lines|^{2/3})$ incidences, but the constant hidden by the $O(\cdot)$-notation is not optimal. 
For example, all lines in this section have positive slopes. 
In Section \ref{sec:MainProof}, we consider the general case and sets of lines that optimize the constant.

For positive integers $s$ and $t$, we define $(s,t) = \GCD(s,t)$.
In other words, $(s,t)$ is the largest positive integer that divides both $s$ and $t$.

\begin{figure}[ht]
    \centering
    \begin{subfigure}[b]{0.25\textwidth}
    \centering
        \includegraphics[width=0.97\textwidth]{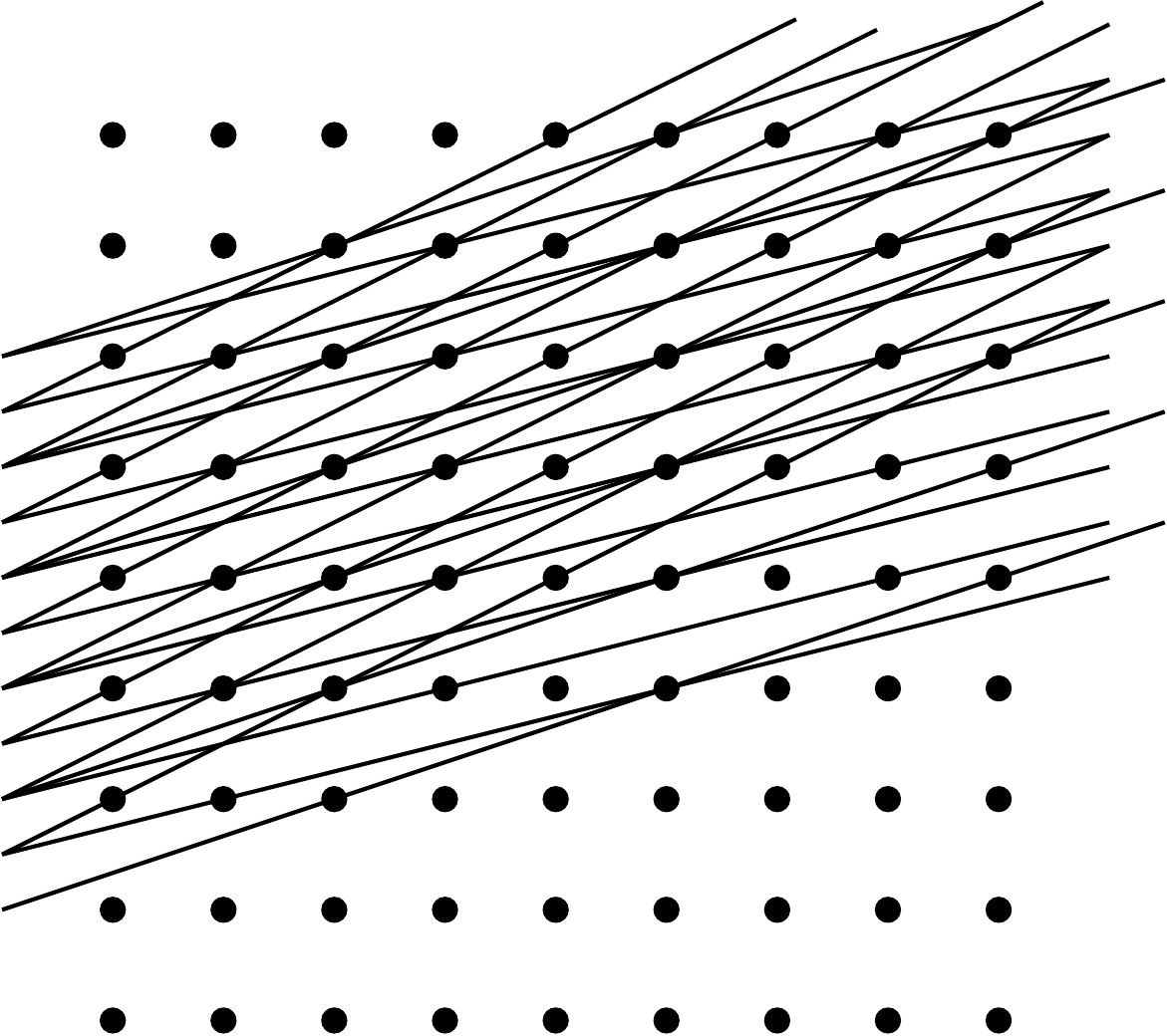}
    \end{subfigure}
    \hspace{1cm}
    \begin{subfigure}[b]{0.16\textwidth}
        \centering
        \includegraphics[width=\textwidth]{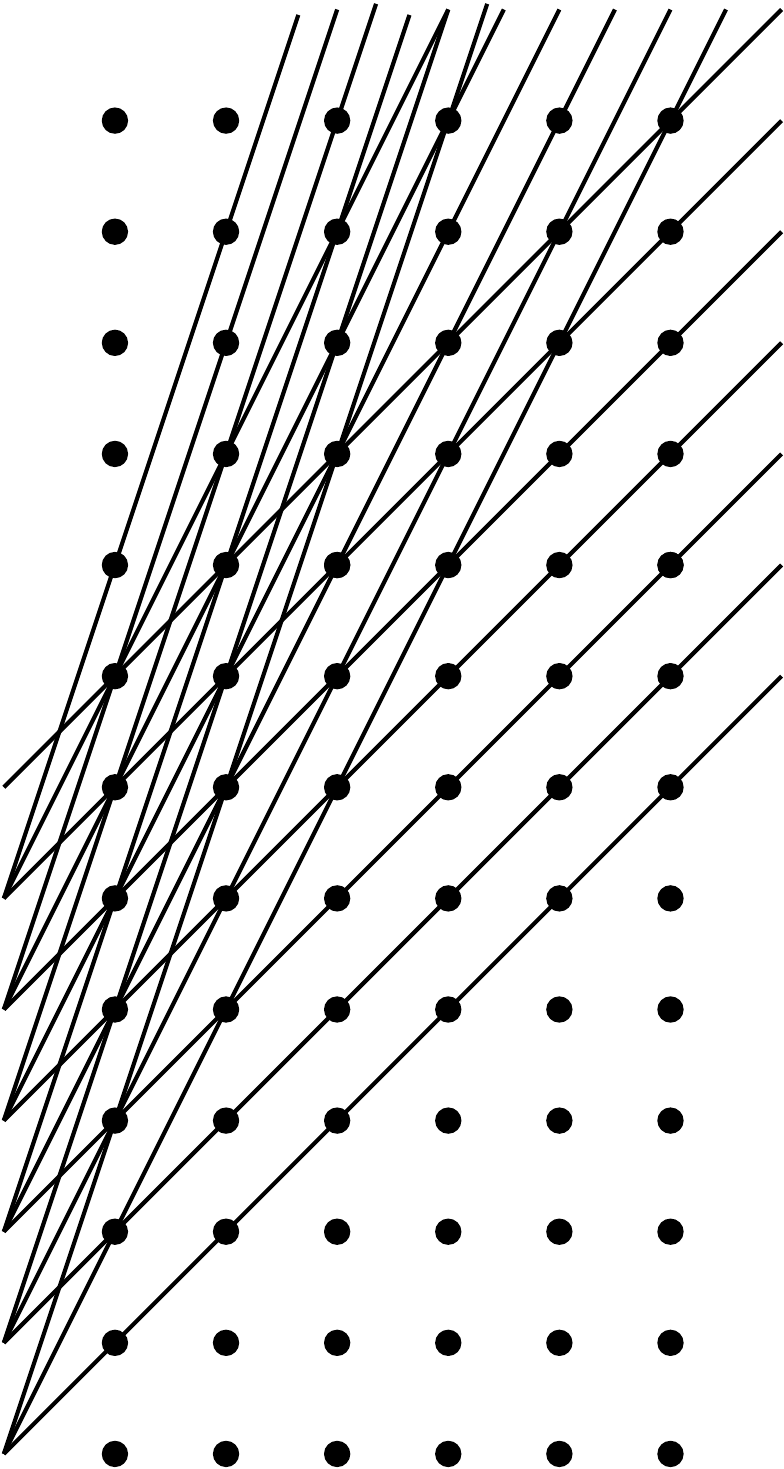}
    \end{subfigure}
    \vspace{2mm}
    \caption{Left: Erd\H os's construction. Right: Elekes's construction.}
    \label{fi:OriginalConstructions}
\end{figure}

\parag{Erd\H os's construction.} 
In Erd\H os's construction, the point set is a $\sqrt{n}\times\sqrt{n}$ section of the integer lattice:
\[ \pts = \{ (i,j) \in \NN^2\ :\ 1\le i,j \le \sqrt{n} \}. \]
The line slopes are rational numbers between 0 and 1, where the numerator and denominator are not too large.
In particular, the set of lines is
\begin{align*} 
\lines = \Big\{y=\frac{a}{b}(x-c)+d\ :\ a,b\in \NN,\ (a,b)=1,\ &1\le a < b \le n^{1/6}, \\
&1\le c\le b,\ \sqrt{n}/2 \le d \le 3\sqrt{n}/4 \Big\}. 
\end{align*}
After fixing the slope of a line (that is, fixing $a,b$), the parameters $c$ and $d$ define a family of $b\sqrt{n}/4$ lines with this slope. 
To see why $|\lines|\approx n$, we require tools that are introduced in Section \ref{sec:totient}.
This configuration is depicted in Figure \ref{fi:OriginalConstructions}.

\parag{Elekes's construction.} 
In Elekes's construction, the point set is a tall and thin section of the integer lattice:
\[ \pts = \{ (i,j) \in \NN^2\ :\ 1\le i \le n^{1/3},\ 1\le j \le n^{2/3} \}. \]
The line slopes are the integers between 1 and $n^{1/3}$.
For each slope, there are $n^{2/3}$ lines with that slope. 
In particular, the set of lines is
\[ \lines = \{ y=ax+b\ :\ a,b\in \NN,\ 1\le a \le n^{1/3},\ 1\le b \le n^{2/3}\}. \]
Unlike Erd\H os's construction, in this case it is easy to see that $|\lines| = n$.
This configuration is depicted in Figure \ref{fi:OriginalConstructions}.

\parag{The infinite family of constructions.} 
In Erd\H os's construction, the point set is a $\sqrt{n}\times \sqrt{n}$ section of the integer lattice.
In Elekes's construction, the point set is an $n^{1/3}\times n^{2/3}$ section of the integer lattice.
These are the two extreme cases of the following family.
We fix $1/3 < \alpha < 1/2$ and consider the point set
\[ \pts = \{ (i,j) \in \NN^2\ :\ 1\le i \le n^\alpha,\ 1\le j \le n^{1-\alpha} \}. \]

In Erd\H os's construction, the line slopes are rational numbers between 0 and 1. 
In Elekes's construction, the line slopes are integers between 1 and $n^{1/3}$.
We now mix these two cases, defining 
\begin{align*} 
\lines = \bigg\{y=&\left(a+\frac{b}{c}\right)(x-i)+d\ :\ a,b,c,d,i\in \NN,\ 0\le a < n^{1-2\alpha}/4, \\
&\hspace{10mm} 1\le b< c\le n^{\alpha-1/3},\ (b,c)=1,\ 0\le i < c,\ n^{1-\alpha}/2 \le d < 3n^{1-\alpha}/4\bigg\}.
\end{align*}
As $\alpha$ approaches $1/2$, this set of lines approaches the set from Erd\H os's construction.
As $\alpha$ approaches $1/3$, this set of lines approaches the set from Elekes's construction.
Thus, we can indeed consider the constructions of Erd\H os and Elekes as the endpoints of the infinite family of constructions.

\section{Preliminaries: Euler's totient function} \label{sec:totient}

For an integer $j>0$, we recall that \emph{Euler's totient function} of $j$ is
\[ \phi(j) = |\{ i\in \{1,2,\ldots,j-1\}\ :\ (i,j)=1 \}|. \]
Our analysis in Section \ref{sec:MainProof} heavily relies on properties of this totient function and of relatively prime integers.
In this section, we discuss these properties.
Throughout the paper, all logarithms are of base $e$.

\begin{lemma} \label{le:totientProperties}$\quad$\\[2mm]
(a) $\displaystyle \sum_{j=1}^n \phi(j) = \frac{3n^2}{\pi^2} +O(n\log n)$. \\[2mm]
(b) $\displaystyle \sum_{j=1}^n j\cdot \phi(j) = \frac{2n^3}{\pi^2} +O(n^2 \log n)$. \\[2mm]
(c) $\displaystyle \sum_{j=1}^n j^2\cdot \phi(j) = \frac{3n^4}{2\pi^2} +O(n^3 \log n)$. \\[2mm]
(d) $\displaystyle \sum_{j=1}^{n}\frac{\phi(j)}{j} = \frac{6n}{\pi^2} + O(\sqrt{n})$.
\end{lemma}
\begin{proof}
(a) This is a classic formula. For example, see \cite[Theorem 330]{HardyWright79}. \\[2mm]
(b) By relying on part (a) and on the formula for a sum of squares, we obtain that
\begin{align*}
\sum_{j=1}^n j\cdot \phi(j) &= n\cdot \sum_{j=1}^n \phi (j) - \left(\sum\limits_{j=1}^1 \phi (j) + \sum\limits_{j=1}^2 \phi (j) + \sum\limits_{j=1}^3 \phi (j) + \cdots + \sum\limits_{j=1}^{n-1} \phi (j)\right) \\[2mm]
&=\frac{3}{\pi^2}\cdot \left(n^3-\left(1^2+2^2+\cdots +(n-2)^2+(n-1)^2 \right)+ n\cdot O(n \log n)\right) \\[2mm]
&= \frac{3}{\pi^2}\cdot\left(n^3-\frac{(n-1)\cdot n\cdot (2n-1)}{6}\right) +O(n^2\log n)= \frac{2n^3}{\pi^2} + O(n^2 \log n).
\end{align*}
(c) By relying on part (b) and on the formula for a sum of cubes, we obtain that
\begin{align*}
\sum_{j=1}^n j^2\cdot \phi(j) &= n\cdot \sum_{j=1}^n j\cdot \phi (j) - \left(\sum\limits_{j=1}^1 j\cdot\phi (j) + \sum\limits_{j=1}^2 j\cdot\phi (j) + \sum\limits_{j=1}^3 j\cdot\phi (j) + \cdots + \sum\limits_{j=1}^{n-1} j\cdot\phi (j)\right) \\[2mm]
&=\frac{2}{\pi^2}\cdot \left(n^4-\left(1^3+2^3+\cdots +(n-2)^3+(n-1)^3 \right) + n\cdot O(n^2 \log n) \right) \\[2mm]
&= \frac{2}{\pi^2}\cdot\left(n^4-\frac{n^2\cdot (n-1)^2}{4}\right) +O(n^3\log n)= \frac{3n^4}{2\pi^2} + O(n^3 \log n).
\end{align*}
(d) We note that 
\begin{align*}
\sum\limits_{j=1}^{n}\phi(j) = n\cdot \sum\limits_{j=1}^{n}\frac{\phi(j)}{j} - \sum_{i=1}^{n-1}\sum_{j=1}^i \frac{\phi(j)}{j}.
\end{align*}
Combining this with part (a), dividing by $n$, and rearranging, leads to 
\begin{equation} \label{eq:r/phi(r)aux}
\sum_{j=1}^{n}\frac{\phi(j)}{j} = \frac{3n}{\pi^2} +O(\log n) + \frac{1}{n}\cdot\sum_{i=1}^{n-1}\sum_{j=1}^i \frac{\phi(j)}{j}.
\end{equation}

We now prove by induction on $n$ that $\sum_{j=1}^{n}\frac{\phi(j)}{j} \le \frac{6n}{\pi^2} + C\cdot \sqrt{n}$, where $C$ is a sufficiently large constant.
For the induction basis, the case where $n\le 100$ holds when $C$ is sufficiently large. 
For the induction step, consider $n>100$ and assume that the claim holds for every smaller value of $n$.
Combining this assumption with \eqref{eq:r/phi(r)aux} implies that 
\begin{align*}
\sum_{j=1}^{n}\frac{\phi(j)}{j} &\le \frac{3n}{\pi^2} +O(\log n) + \frac{1}{n}\cdot\sum_{i=1}^{n-1}\left(\frac{6i}{\pi^2} + C\cdot \sqrt{i}\right)\\[2mm]
&=\frac{3n}{\pi^2} +O(\log n) + \frac{1}{n}\cdot\left(\frac{6n(n-1)}{2\pi^2} + C\cdot\sum_{i=1}^{n-1}\sqrt{i}\right)\\[2mm]
&= \frac{3n}{\pi^2} +O(\log n) + \frac{3(n-1)}{\pi^2} + \frac{C}{n}\cdot \sum_{i=1}^{n-1}\sqrt{i}.
\end{align*}

Standard integration shows that $\sum_{i=1}^{n}\sqrt{i} = 2n^{3/2}/3 + O(n)$.
Combining this with the above leads to   
\begin{align*}
\sum_{j=1}^{n}\frac{\phi(j)}{j} &< \frac{6n}{\pi^2} +O(\log n) + \frac{C}{n}\cdot \left( 2n^{3/2}/3 + O(n)\right) = \frac{6n}{\pi^2} + \frac{2Cn^{1/2}}{3} +O(C+\log n).
\end{align*}
When $C$ is sufficiently large with respect to the constants hidden by the $O(\cdot)$-notation, we have that $ \frac{2Cn^{1/2}}{3} + O(C+\log n) < Cn^{1/2}$.
This completes the induction step.
A symmetric inductive argument implies that $\sum_{j=1}^{n}\frac{\phi(j)}{j} \ge \frac{6n}{\pi^2} - D\cdot \sqrt{n}$, where $D$ is a sufficiently large constant.
\end{proof}

Studying $\phi(j)$ tells us how many numbers are relatively prime to $j$, between 1 and $j$.
Sometimes, we are not interested in the amount of relatively prime numbers, but rather in the sum of these numbers.

\begin{lemma} \label{le:TotientPastN}
$\quad$ \\[2mm]
(a) $\displaystyle \sum_{s<n \atop (s,n)=1} s = \frac{n\cdot \phi(n)}{2}$. \\[2mm]
(b) For a positive integer $k$, the number of integers between 1 and $nk$ that are relatively prime to $n$ is $k\cdot \phi(n)$. \\[2mm]
(c) Every positive integer $k$ satisfies that $\displaystyle \sum_{s<kn \atop (s,n)=1} s = k^2\cdot \frac{n\cdot \phi(n)}{2}$.
\end{lemma}
\begin{proof}
(a) We note that $(s,n)=1$ if and only if $(n-s,n)=1$.
By pairing up $s$ with $n-s$, we obtain $\phi(n)/2$ pairs of numbers that sum up to $n$.
The sum of all numbers from all pairs is $n\cdot\phi(n)/2$. \\[2mm]
(b) Every $1< i< n$ satisfies $(i,n)=(i+n,n)$.
In particular, $(i,n)=1$ if and only if $(n,i+n)=1$.
This implies that, for every $j\in \NN$, the set $\{jn+1,jn+2,\ldots,jn+n\}$ contains exactly $\phi(n)$ numbers that are relatively prime to $n$.  
This immediately leads to the assertion of part (b). \\[2mm]
(c) By part (b) of the current lemma, the number of integers between $(j-1)n$ and $jn$ that are relatively prime to $n$ is $jn\cdot \phi(n)-(j-1)n \cdot \phi(n) = \phi(n)$. 
Each of these integers is larger than its corresponding element from part (a) by $(j-1)n$.
This leads to
\[ \sum_{(j-1)n<s<jn \atop (s,n)=1} s = \frac{n\cdot \phi(n)}{2} + (j-1) n \cdot \phi(n) = \frac{(2j-1)\cdot n\cdot \phi(n)}{2}. \]

We conclude that
\begin{align*}
\sum_{s<kn \atop (s,n)=1} s &= \sum_{j=1}^{k}\sum_{(j-1)n<s<jn \atop (s,n)=1} s = \sum_{j=1}^{k} \frac{(2j-1)\cdot n\cdot \phi(n)}{2} \\
&= \frac{n\cdot \phi(n)}{2} \cdot \sum_{j=1}^{k} (2j-1) = \frac{n\cdot \phi(n)}{2} \cdot \frac{k(1+(2k-1))}{2} = k^2 \cdot \frac{n\cdot \phi(n)}{2}.
\end{align*}
\end{proof}

We also study numbers that are relatively prime to $n$ in sets other than $\{1,2,3,\ldots,n\}$.
For integers $1 < m \leq n$, we define 
\[ \phi_m (n) = |\{1\le a \le m\ :\ (a, n) = 1\}|. \]
For a non-negative $x \in \RR$, let $\langle x \rangle = x - \lfloor x \rfloor$. 
In other words, $\langle x \rangle$ is the part of $x$ to the right of the decimal point.
A statement of the form $a|b$ means ``$a$ divides $b$.''
For a positive integer $n$, we let $w(n)$ denote the number of prime divisors of $n$.

\begin{lemma} \label{le:phiMN}
Any integers $1< m\le n$ satisfy that
\[ \phi_m(n) = \frac{m}{n}\cdot \phi(n) + O(2^{w(n)}).  \]
\end{lemma}
\begin{proof}

We write $w=w(n)$ and let the prime factorization of $n$ be 
\[ n = {p_1}^{l_1} \cdot {p_2}^{l_2} \cdot ... \cdot {p_w}^{l_w}. \]

For an integer $i$, we set $A_i = \{1\le a \le m\ :\ i | a\}$. 
Note that $|A_i| = \lfloor\frac{m}{i}\rfloor$.
Since 0 does not divide any nonzero integer, we have that $A_0 = \emptyset$. 
Since 1 divides every integer, we have that $A_1 = \{1,2,\ldots,m\}$.
For an integer $1\le k \le w$, let ${[w] \choose k}$ be the set of subsets of $\{1,2,\ldots,w\}$ of size $k$.
By the inclusion-exclusion principle,
\[ \phi_m(n) = m - |\cup_{i=1}^{w}A_{p_i}| = m + \sum\limits_{k=1}^{w}(-1)^{k}\cdot \sum_{S \in {[w] \choose k}}|\cap_{i\in S} A_{p_i}|. \]

For a fixed $k$ and $S \in {[w] \choose k}$, we set $q_S = \prod_{i\in S} p_i$.
Note that  $\cap_{i\in S} A_{p_i} = A_{q_S}$.
We may thus rewrite the above as
\[ \phi_m(n) = m +\sum\limits_{k=1}^w (-1)^{k} \cdot \sum\limits_{S \in {[w] \choose k}} |A_{q_S}| = m +\sum\limits_{k=1}^w (-1)^{k} \cdot \sum\limits_{S \in {[w] \choose k}} \left\lfloor\frac{m}{q_S}\right\rfloor. \]

We define ${[w] \choose 0}=\{\emptyset\}$ and $q_\emptyset=1$, which allows us to rewrite the above as
\[ \phi_m(n) =  \sum\limits_{k=0}^w (-1)^{k} \cdot \sum\limits_{S \in {[w] \choose k}} \left\lfloor\frac{m}{q_S}\right\rfloor = \sum\limits_{k=0}^w (-1)^{k} \cdot \sum\limits_{S \in {[w] \choose k}} \frac{m}{q_S} - \sum\limits_{k=0}^w (-1)^{k} \cdot \sum\limits_{S \in {[w] \choose k}} \left\langle\frac{m}{q_S}\right\rangle. \]
The first term on the right-hand side is $m\cdot(1-\frac{1}{p_1})\cdots(1-\frac{1}{p_w})$. The second term is the decimal part that was removed by the floors. 
Next, we note that
\begin{align*} 
\phi_m(n) &= \sum\limits_{k=0}^w (-1)^{k} \sum\limits_{S \in {[w] \choose k}} \frac{m}{q_S} - \sum\limits_{k=1}^w (-1)^{k} \sum\limits_{S \in {[w] \choose k}} \left\langle\frac{m}{q_S}\right\rangle \\
&\leq m\left(1-\frac{1}{p_1}\right)\cdots\left(1-\frac{1}{p_w}\right) + \sum\limits_{1\le k \le w \atop k \text{ odd}} \sum\limits_{S \in {[w] \choose k}} \left\langle\frac{m}{q_S}\right\rangle.
\end{align*}

By definition, every term of the form $\langle m/q_S\rangle$ is smaller than 1. 
The number of subset of $\{1,2,\ldots,w\}$ of an odd size is $2^{w-1}$.
We thus obtain that 
\[\phi_m(n) < m\left(1-\frac{1}{p_1}\right)\cdots\left(1-\frac{1}{p_w}\right) + 2^{w-1}.\]

By another inclusion-exclusion principle, we have that
\[ \phi(n) = n\left(1-\frac{1}{p_1}\right)\left(1-\frac{1}{p_2}\right)\cdots\left(1-\frac{1}{p_w}\right). \]
Thus, 
\begin{equation*} 
\phi_m(n) < \frac{m}{n}\cdot \phi(n) + 2^{w-1}.
\end{equation*}

A symmetric argument with the even values of $k$ implies that
\begin{equation*} 
\phi_m(n) > \frac{m}{n}\cdot \phi(n) - 2^{w-1}.
\end{equation*}
\end{proof}

We next generalize Lemma~\ref{le:phiMN} to every value of $m$.

\begin{corollary} \label{co:phiMN}
Any integers $m,n >1$ satisfy that
\[ \phi_m(n) = \phi(n)\cdot \frac{m}{n} + O(2^{w(n)}).  \]
\end{corollary}
\begin{proof}
By Lemma~\ref{le:TotientPastN}(b), the number of integers relatively prime to $n$ in $\{1,\ldots,n\cdot \lfloor m/n\rfloor\}$ is $\phi(n)\cdot \lfloor m/n\rfloor$.
By the proof Lemma~\ref{le:TotientPastN}(b), the number of integers relatively prime to $n$ in $\{n\cdot \lfloor m/n\rfloor+1, \ldots, m\}$ is the same as in $\{1,\ldots,\langle m/n \rangle \cdot n\}$.
Combining this with Lemma~\ref{le:phiMN} implies that
\begin{align*} 
\phi_m(n) = \phi(n)\cdot \lfloor m/n\rfloor + \phi_{\langle m/n \rangle \cdot n}(n) &< \phi(n)\cdot \lfloor m/n\rfloor + \phi(n)\cdot \langle m/n \rangle + O(2^{w(n)}) \\[2mm]
&= \phi(n)\cdot \frac{m}{n} + O(2^{w(n)}). 
\end{align*}
\end{proof}

When applying Lemma~\ref{le:phiMN} and Corollary~\ref{co:phiMN}, we get expressions of the form $2^{w(r)}$.
We require the following result to sum up such expressions.

\begin{lemma} \label{le:sum2w(r)}
Every positive integer $m$ satisfies the following: \\
(a) $\displaystyle \sum_{r=2}^m 2^{w(r)} = O(m\log\log m)$.\\
(b) $\displaystyle \sum_{r=2}^m r 2^{w(r)} = O(m^2\log\log m)$.
\end{lemma}
\begin{proof}
(a) Euler's famous formula states that 
\[ \sum_{p\le m \atop p \text{ is prime}} \frac{1}{p} = \log \log m + O(1). \]
Let $\pi(m)$ be the number of primes smaller or equal to $m$. 
The above implies that
\begin{align*}
\frac{1}{m}\sum_{r=2}^m w(r) &= \frac{1}{m}\sum_{r=2}^m \sum_{p | r \atop p \text{ prime}} 1= \frac{1}{m}\sum_{2\le p \le m \atop p \text{ prime}} \sum_{2\le r \le m \atop p|r} 1 = \frac{1}{m}\sum_{2\le p \le m \atop p \text{ prime}}\left\lfloor\frac{m}{p}\right\rfloor \\[2mm]
&=\frac{1}{m}\sum_{2\le p \le m \atop p \text{ prime}}\left(\frac{m}{p}+O(1)\right) =\frac{1}{m}\sum_{2\le p \le m \atop p \text{ prime}}\frac{m}{p}+O\left(\frac{1}{m}\sum_{2\le p \le m \atop p \text{ prime}}1\right) \\[2mm]
&= \sum_{p\le m \atop p \text{ prime}} \frac{1}{p} + O\left(\frac{\pi(m)}{m}\right) =\log \log m + O(1).
\end{align*}
We may rearrange this as
\begin{equation} \label{eq:sumW(r)}
\sum_{r=1}^m w(r) = m\log \log m + O(m).
\end{equation}

Since an integer $s$ has $O(\log s / \log\log s)$ prime divisors (for example, see \cite[Section 22.10]{HardyWright06}), every $2\le r \le m$ satisfies that $w(r)= O(\log m / \log \log m)$.
We wish to maximize the expression $\sum_{r=2}^m 2^{w(r)}$ under the conditions $w(r)= O(\log m / \log \log m)$ and \eqref{eq:sumW(r)}.
Since exponential functions are convex, this maximum is obtained when $\Theta(m(\log\log m)^2/\log m)$ values of $r$ satisfy $w(r) = \Theta(\log m / \log \log m)$ and the rest satisfy $w(r)=0$.
We conclude that
\[ \sum_{r=2}^m 2^{w(r)} = O\left(\frac{m(\log\log m)^2}{\log m}\cdot \frac{\log m}{\log\log m}\right) = O\left(m\log\log m\right).\]
 
(b) By part (a), we have that
\[ \sum_{r=2}^m 2^{w(r)}r \le m\sum_{r=2}^m 2^{w(r)} = O\left(m^2\log\log m\right).\]
\end{proof}


\section{Proof of Theorem~\ref{th:main}} \label{sec:MainProof}

We fix $1/3\le \alpha\le 1/2$ and consider the point set
\[ \pts = \{(i,j) \in \NN^2\ :\ 1\le i \le n^{\alpha},\ 1\le j \le n^{1-\alpha} \}. \]

We fix $0<\eps<1$ such that $\eps = o(1)$ and $\eps  = \omega(n^{-\alpha})$.
In other words, $\eps$ decreases as $n$ increases, and the rate of decrease is asymptotically slower than $n^{-\alpha}$.
At the end of this section, we show how the choice of $\eps$ controls the relation between $|\pts|$ and $|\lines|$.
For simplicity, we assume that both $n^{\alpha}$ and $\eps n^{\alpha}$ are integers. 
Let $\lines$ be the set of all lines that contain at least $\eps n^{\alpha}$ points of $\pts$, are not axis parallel, and are not of slope $\pm n^{1-2\alpha}$.

We partition $\lines$ into four subsets according the their slope $s$:
\begin{itemize}[noitemsep,topsep=1pt]
\item $\lines_1$: lines with slope $0<s<n^{1-2\alpha}$,
\item $\lines_2$: lines with slope $s>n^{1-2\alpha}$,
\item $\lines_3$: lines with slope $-n^{1-2\alpha}<s<0$,
\item $\lines_4$: lines with slope $s<-n^{1-2\alpha}$.
\end{itemize}

\begin{figure}[ht]
    \centering
    \begin{subfigure}[b]{0.34\textwidth}
    \centering
        \includegraphics[width=0.97\textwidth]{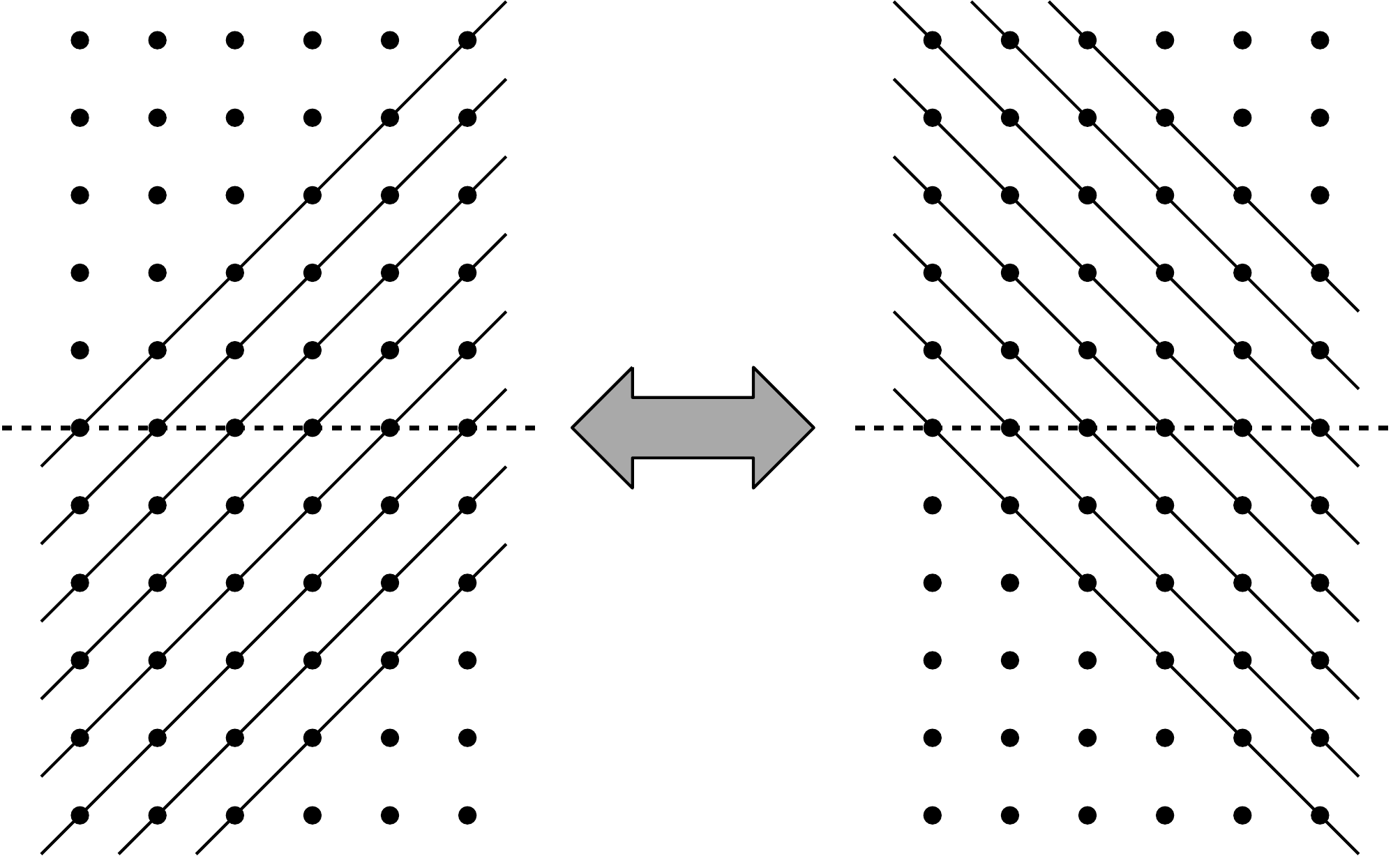}
        \caption{}
    \end{subfigure}
    \hspace{1cm}
    \begin{subfigure}[b]{0.19\textwidth}
        \centering
        \includegraphics[width=\textwidth]{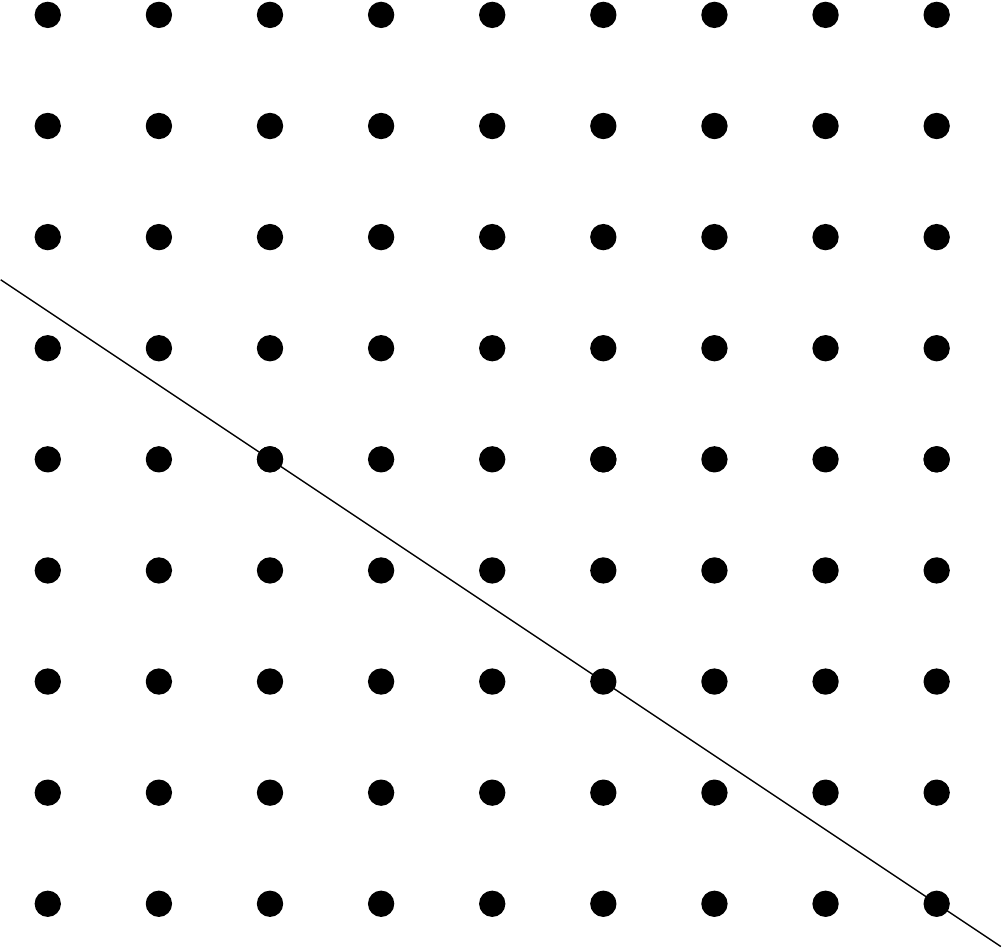}
		\caption{}    
    \end{subfigure}
    \hspace{1cm}
    \begin{subfigure}[b]{0.14\textwidth}
        \centering
        \includegraphics[width=\textwidth]{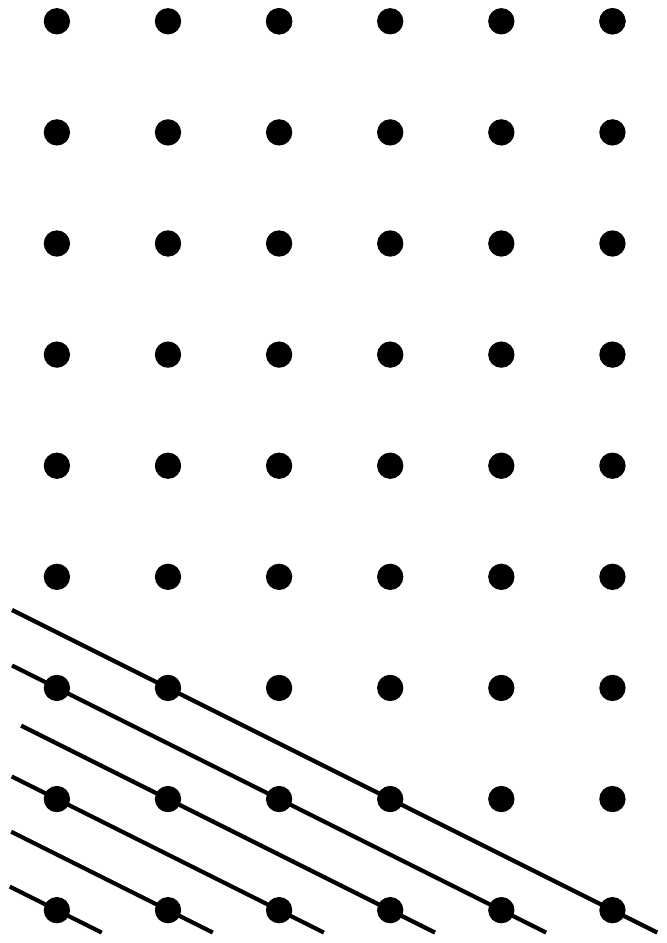}
		\caption{}    
    \end{subfigure}
    \vspace{2mm}
    \caption{(a) The reflection of the plane across $y=(n^{1-\alpha}+1)/2$ is a bijection between $\lines_1$ and $\lines_3$ and takes $\pts$ to itself.  (b) A line with slope -2/3 contains a point every other row and every third column. (c) The line set $\lines_{s/r}$.}
    \label{fi:LineSets}
\end{figure}

We claim that $|\lines_1| = |\lines_3|$ and that $I(\pts,\lines_1)=I(\pts,\lines_3)$. 
Indeed, the reflection of the plane across the line $y=(n^{1-\alpha}+1)/2$ is a bijection between $\lines_1$ and $\lines_3$ and takes $\pts$ to itself. 
See Figure \ref{fi:LineSets}(a).
A symmetric argument shows that $|\lines_2| = |\lines_4|$ and that $I(\pts,\lines_2)=I(\pts,\lines_4)$.

When $\alpha=1/2$, all four sets of lines are identical.
Indeed, in this case a 90 degrees rotation around the center of $\pts$ is a bijection between $\lines_1$ and $\lines_4$ (and takes $\pts$ to itself).  
This avoids some of the main difficulties of our analysis, making Erd\H os's construction easier to study. 
In particular, this case does not require Lemma~\ref{le:totientProperties}(d), Lemma~\ref{le:phiMN}, Corollary~\ref{co:phiMN}, and Lemma~\ref{le:sum2w(r)}.

\parag{The bottom rows with $\lines_3$.} 
We first study the lines of $\lines_3$, where the slopes are between $-n^{1-2\alpha}$ and 0. 
Each such slope can be uniquely written as $-s/r$, where $r,s\in \NN$, $s<r\cdot n^{1-2\alpha}$, and $(r,s)=1$.
We fix such values of $r$ and $s$.
Recall that $\pts$ consists of $n^{\alpha}$ columns and of $n^{1-\alpha}$ rows.
A line with slope $-s/r$ is incident to at most one point every $r$ columns and to at most one point every $s$ rows.
See Figure \ref{fi:LineSets}(b). 
We may assume that $r\le 1/\eps$, since otherwise no line with slope $-s/r$ contains $\eps n^\alpha$ points of $\pts$.
Since $\eps  = \omega(n^{-\alpha})$, we get that $r=o(n^\alpha)$. 
A symmetric argument leads to $s\le n^{1-2\alpha}/\eps=o(n^{1-\alpha})$.

Let $\lines_{s/r}$ be the set of $n^{\alpha}$ lines of slope $-s/r$ that are incident to a point on the bottom row of $\pts$.
See Figure \ref{fi:LineSets}(c).
Some lines of $\lines_{s/r}$ are not in $\lines_3$, since they are incident to fewer than $\eps n^\alpha$ points of $\pts$.
The $r$ leftmost lines of $\lines_{s/r}$ contain a single point of $\pts$.
The next $r$ leftmost lines of $\lines_{s/r}$ are incident to two points of $\pts$. 
The next $r$ leftmost lines are incident to three points, and so on.
After getting to $r$ lines that are incident to $\lceil n^\alpha/r\rceil-1$ points of $\pts$, at most $r$ lines contain $\lceil n^\alpha/r\rceil$ points of $\pts$.
(This relies on $s<r\cdot n^{1-2\alpha}$. As shown below for $\lines_4$, when $s>r\cdot n^{1-2\alpha}$ we have a rather different behavior.)

To find $|\lines_{s/r}\cap \lines_3|$, we remove from $\lines_{s/r}$ the lines that contain fewer than $\eps n^\alpha$ points of $\pts$. 
By the preceding paragraph, we have that
\begin{equation} \label{eq:LinesBottomRow25}
|\lines_{s/r}\cap\lines_3| = n^\alpha - r(\eps n^\alpha-1) = n^\alpha(1-\eps r) +r. 
\end{equation}
We also have that 
\begin{align} 
I(\pts,\lines_{s/r}\cap\lines_3) &=r\big(\eps n^\alpha+(\eps n^\alpha+1)+(\eps n^\alpha+2)+\cdots +n^\alpha/r - O(n^\alpha/r)\big)  \nonumber \\[2mm]
&=r\frac{(n^\alpha/r+\eps n^\alpha)(n^\alpha/r-\eps n^\alpha)}{2} - O(n^\alpha) \nonumber \\
&= \frac{n^{2\alpha}}{2r}-\frac{r\eps^2n^{2\alpha}}{2}- O(n^{\alpha}). \label{eq:IncidencesBottomRow25}
\end{align}

Let $\lines_{s/r}^*$ be the set of lines with slope $-s/r$ that are incident to at least one point from the $s$ bottom rows of $\pts$.  
A line with slope $-s/r$ contains at most one point from every $s$ consecutive rows of $\pts$, so each point from the $s$ bottom rows is incident to a distinct line with slope $-s/r$.
This implies that $|\lines_{s/r}^*|=s n^\alpha$.
Also, \eqref{eq:LinesBottomRow25} implies that 
\begin{equation} \label{eq:LinesBottomRows25}
|\lines_{s/r}^*\cap\lines_3| \le s n^\alpha(1-\eps r) +rs. 
\end{equation}

\begin{figure}[ht]
    \centering
    \begin{subfigure}[b]{0.12\textwidth}
    \centering
        \includegraphics[width=0.97\textwidth]{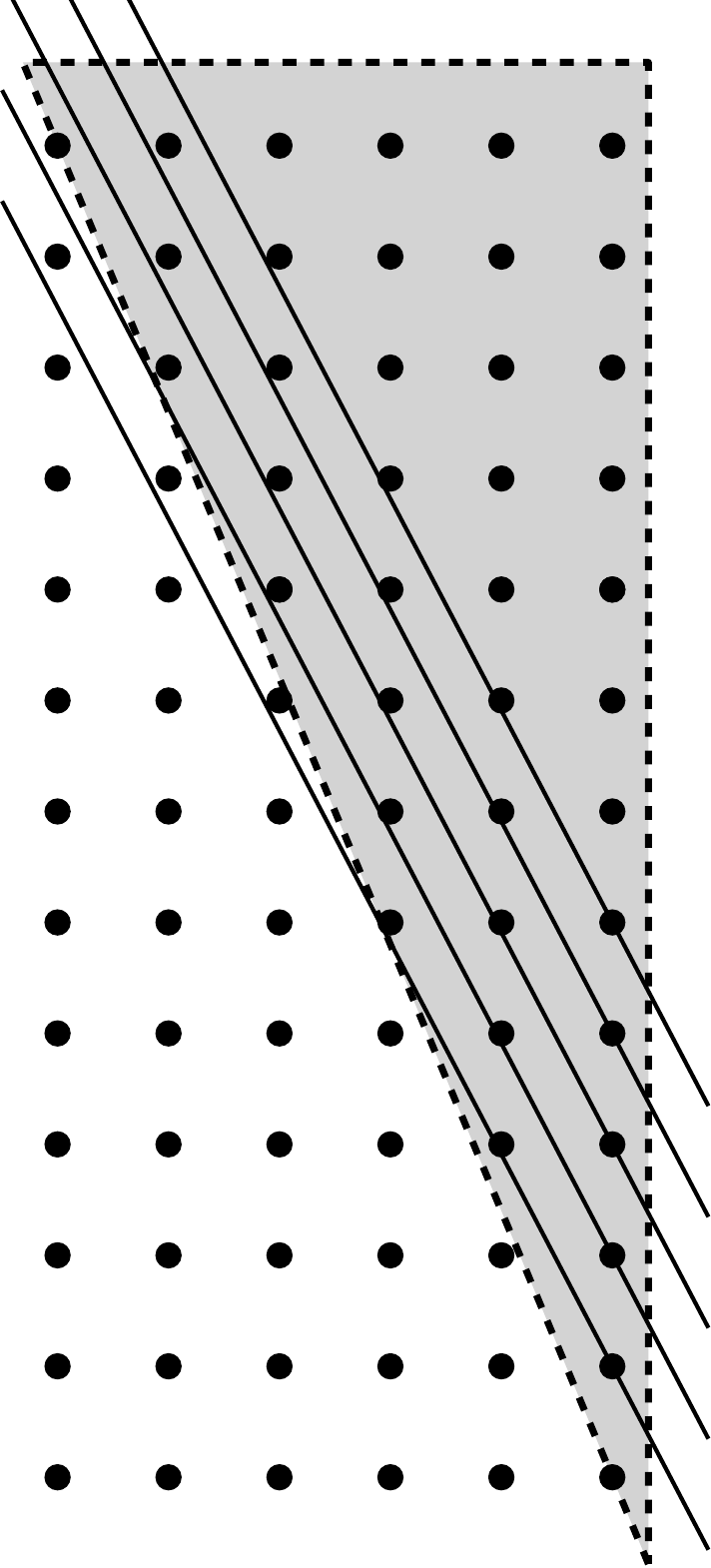}
        \caption{}
    \end{subfigure}
    \hspace{1cm}
    \begin{subfigure}[b]{0.14\textwidth}
        \centering
        \includegraphics[width=\textwidth]{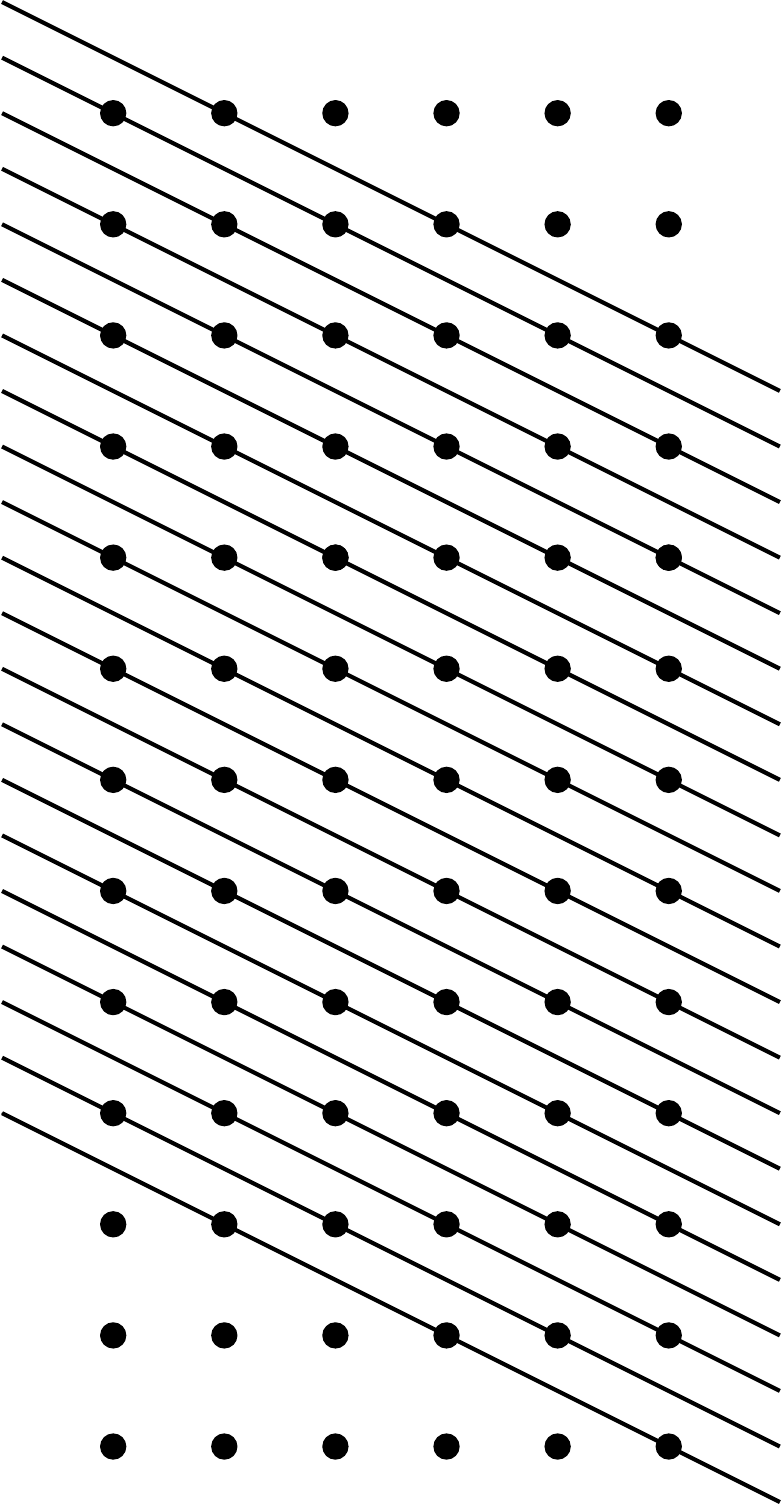}
		\caption{}    
    \end{subfigure}
    \vspace{2mm}
    \caption{(a) Only lines that are incident to a point inside the triangle might intersect the top of the triangle.  (b) Many maximal lines may be incident to points of the rightmost column.}
    \label{fi:MoreLineSets}
\end{figure}

We repeat the above analysis for $\lines_{s/r}$ with the second lowest row of $\pts$.
The only difference is that the rightmost line in this row may reach the top of $\pts$ and thus contain one point less than the rightmost line of $\lines_{s/r}$.
This cannot happen to more than one line because $s<r\cdot n^{1-2\alpha}$.
See Figure \ref{fi:MoreLineSets}(a).
Similarly, when studying lines incident to points from the third lowest row, each of the two rightmost lines may contain one point less (compared to the corresponding lines in $\lines_{s/r}$).
For the fourth lowest row, the three rightmost lines may contain one point less, and so on. 
Since a line contains a point every $s$ rows, no line loses more than one point. 
In total, the number of lines that lost one incidence is smaller than $sn^\alpha$.
Moreover, a line that was incident to $\eps n^\alpha$ points in the bottom row may contain $\eps n^\alpha-1$ points in higher rows.
Such a line is not in $\lines_3$, so we do not count any of its incidences.
Each row has at most $r$ such lines, with a total loss of less than $sr\eps n^\alpha$ incidences.
Since $r\le 1/\eps$, this quantity is also at most $sn^{\alpha}$.
Combining this with \eqref{eq:IncidencesBottomRow25} implies that
\begin{equation} \label{eq:IncidencesBottomRows25}
I(\pts,\lines_{s/r}^*\cap\lines_3) \ge s\cdot \left(\frac{n^{2\alpha}}{2r}-\frac{r\eps^2n^{2\alpha}}{2}\right) - 2sn^\alpha. 
\end{equation}

\parag{The rightmost columns with $\lines_3$.} 
We continue with the values of $s$ and $r$ that were fixed above. 
Let $\hat{\lines}_{s/r}$ be the set of $n^{1-\alpha}-s$ lines of slope $-s/r$ that are incident to a point on the rightmost column of $\pts$, excluding the bottom $s$ points.
The top $s$ lines of $\hat{\lines}_{s/r}$ contain a single point of $\pts$.
The next $s$ top lines of $\hat{\lines}_{s/r}$ are incident to two points of $\pts$.
The next $s$ top lines are incident to three points, and so on.
After getting to $s$ lines that are incident to $\lceil n^\alpha/r\rceil-1$ points of $\pts$, some number of lines contain $\lceil n^\alpha/r\rceil$ points of $\pts$.
We refer to a line that contains $\lceil n^\alpha/r\rceil$ points of $\pts$ as a \emph{maximal line}.
Unlike the the case of $\lines_{s/r}$ above, the set $\hat{\lines}_{s/r}$ may contain a lot more than $s$ maximal lines.
See Figure \ref{fi:MoreLineSets}(b).

To find $|\hat{\lines}_{s/r}\cap \lines_3|$, we remove from $\hat{\lines}_{s/r}$ the lines that contain fewer than $\eps n^\alpha$ points of $\pts$. 
Recalling that we ignore the bottom $s$ lines leads to 
\begin{equation} \label{eq:LinesRightColumn25}
|\hat{\lines}_{s/r}\cap\lines_3| = n^{1-\alpha} - s(\eps n^\alpha-1) - s = n^{1-\alpha}- s\eps n^\alpha. 
\end{equation}
A similar argument can be used to estimate the number of maximal lines in $\hat{\lines}_{s/r}$. 
However, it is possible that there are no maximal lines when ignoring the bottom $s$ points on the rightmost column. 
To address this small technicality, we replace $-s$ with $-O(s)$, obtaining that the number of maximal lines is at least  
\[ n^{1-\alpha} - s(\lceil n^\alpha/r\rceil-1) - O(s) = n^{1-\alpha} - s n^\alpha/r - O(s). \]

The above expressions for the amount of maximal lines implies that (the $O(\cdot)$-notation at the top row addresses the removal of the floor function and the redundant $+n^\alpha/r$)
\begin{align} 
I(\pts,\hat{\lines}_{s/r}\cap\lines_3) &=s\cdot \big(\eps n^\alpha+(\eps n^\alpha+1) +(\eps n^\alpha+2)+\cdots +n^\alpha/r - O(n^\alpha/r) \big)  \nonumber \\
&\hspace{50mm} + \left\lceil \frac{n^\alpha}{r}\right\rceil\left(n^{1-\alpha} - s n^\alpha/r - O(s)\right)\nonumber \\
&=s\cdot \frac{(n^\alpha/r+\eps n^\alpha)(n^\alpha/r-\eps n^\alpha)}{2} + \frac{n}{r} - \frac{n^{2\alpha}s}{r^2}- O\left(n^{1-\alpha}\right) \nonumber \\
&= \frac{n^{2\alpha}s}{2r^2}-\frac{s\eps^2n^{2\alpha}}{2} + \frac{n}{r} - \frac{n^{2\alpha}s}{r^2}- O(n^{1-\alpha}) \nonumber \\
&= \frac{n}{r} -\frac{s\eps^2n^{2\alpha}}{2} - \frac{n^{2\alpha}s}{2r^2}- O(n^{1-\alpha}). \label{eq:IncidencesRightColumn25}
\end{align}

Let $\hat{\lines}_{s/r}^*$ be the set of lines with slope $-s/r$ that are incident to a point in the $r$ rightmost columns of $\pts$, but not to a point from the bottom $s$ rows.  
By \eqref{eq:LinesRightColumn25}, we have that 
\begin{equation} \label{eq:LinesRightColumns25}
|\hat{\lines}_{s/r}^*\cap\lines_3| \le r(n^{1-\alpha}- s\eps n^\alpha). 
\end{equation}

As before, when adapting the above analysis of $\hat{\lines}_{s/r}$ to lines that are incident to the other $r-1$ rightmost columns, some lines may lose a single incidence.
Being somewhat wasteful, we bound the number of lines that lose a single incidence as being smaller than $rn^{1-\alpha}$.
Also as before, in each of these $r-1$ columns, at most $s$ lines may no longer belong to $\lines_3$ due to forming $\eps n^\alpha-1$ incidences. 
The number of incidences that are lost due to this is smaller than $sr\eps n^{\alpha}$.
Since $s\le n^{1-2\alpha}/\eps$, this amount is also smaller than $rn^{1-\alpha}$.
Then \eqref{eq:IncidencesRightColumn25} implies that
\begin{equation} \label{eq:IncidencesRightColumns25}
I(\pts,\hat{\lines}_{s/r}^*\cap\lines_3) \ge n -\frac{sr\eps^2n^{2\alpha}}{2} - \frac{n^{2\alpha}s}{2r} - 2rn^{1-\alpha}. 
\end{equation}

\parag{Completing the analysis of $\lines_3$.} 
By the condition about the bottom $s$ rows in the definition of $\hat{\lines}_{s/r}^*$, we have that $\lines_{s/r}^*\cap \hat{\lines}_{s/r}^* = \varnothing$.
Every line of $\lines$ with slope $-s/r$ contains a point from the $s$ bottom rows or the $r$ rightmost columns of $\pts$.
In other words, every line of $\lines$ with slope $-s/r$ is either in $\lines_{s/r}^*$ or in $\hat{\lines}_{s/r}^*$.
By combining \eqref{eq:LinesBottomRows25} and \eqref{eq:LinesRightColumns25}, we get that the number of these lines is at most 
\begin{equation} \label{eq:TotalLinesFixedsr25}
s n^\alpha(1-\eps r) +rs + r(n^{1-\alpha}- s\eps n^{\alpha}) = s \cdot (n^\alpha-2 n^\alpha \eps r +r) + rn^{1-\alpha}.
\end{equation}
By combining \eqref{eq:IncidencesBottomRows25} and \eqref{eq:IncidencesRightColumns25}, we get that the number of incidences with those lines is at least
\begin{align} 
s\cdot \left(\frac{n^{2\alpha}}{2r}-\frac{r\eps^2n^{2\alpha}}{2}\right) - 2sn^\alpha + n &-\frac{sr\eps^2n^{2\alpha}}{2} - \frac{n^{2\alpha}s}{2r} - 2rn^{1-\alpha} \nonumber \\
&= n - rs\eps^2n^{2\alpha} -o(n). \label{eq:TotalIncidencesFixedsr25}
\end{align}

We recall that $0<r<1/\eps$, that $1\le s<rn^{1-2\alpha}$, and that $\gcd(r,s)=1$.
Combining this with \eqref{eq:TotalLinesFixedsr25} and Lemma~\ref{le:TotientPastN} implies that 
\begin{align*} 
|\lines_3| &\le \sum_{r=1}^{1/\eps} \sum_{1\le s < rn^{1-2\alpha} \atop (r,s)=1} \left(s \cdot (n^\alpha-2 n^\alpha \eps r +r) + rn^{1-\alpha}\right) \\
&= \sum_{r=1}^{1/\eps} \left( n^{2-4\alpha}\cdot\frac{r\cdot \phi(r)}{2} \cdot (n^\alpha-2 n^\alpha \eps r +r) + n^{1-2\alpha}\phi(r)\cdot rn^{1-\alpha}\right) \\
&= \sum_{r=1}^{1/\eps} \left( r\cdot \phi(r) \cdot \frac{3n^{2-3\alpha}}{2} + r^2\cdot \phi(r) \cdot \frac{n^{2-4\alpha}-2\eps n^{2-3\alpha}}{2}\right). 
\end{align*}

By applying parts (b) and (c) of Lemma~\ref{le:totientProperties}, we obtain that
\begin{align} 
|\lines_3| &\le \left(\frac{2}{\pi^2\eps^3} +o(1/\eps^3)\right) \cdot \frac{3n^{2-3\alpha}}{2} + \left(\frac{3}{2\pi^2\eps^4} +o(1/\eps^4)\right) \cdot \frac{n^{2-4\alpha}-2\eps n^{2-3\alpha}}{2} \nonumber \\[2mm] 
&= \frac{3n^{2-3\alpha}}{\pi^2\eps^3}-\frac{3n^{2-3\alpha}}{2\pi^2\eps^3}+o\left(\frac{n^{2-3\alpha}}{\eps^3}\right) = \frac{3n^{2-3\alpha}}{2\pi^2\eps^3} +o\left(\frac{n^{2-3\alpha}}{\eps^3}\right). \label{eq:LinesL325}
\end{align}

Combining \eqref{eq:TotalIncidencesFixedsr25} and Lemma~\ref{le:TotientPastN} leads to
\begin{align*} 
I(\pts,\lines_3) &= \sum_{r=1}^{1/\eps} \sum_{1\le s < rn^{1-2\alpha} \atop (r,s)=1} \left(n - rs\eps^2n^{2\alpha} -o(n))\right) \\
&= \sum_{r=1}^{1/\eps} \left(n^{1-2\alpha}\cdot \phi(r)\cdot n - n^{2-4\alpha}\cdot\frac{r\cdot \phi(r)}{2}\cdot r\eps^2n^{2\alpha} -o\left(n^{1-2\alpha}\cdot\phi(r)\cdot n\right)\right) \\
&= \sum_{r=1}^{1/\eps} \left(\phi(r)\cdot (n^{2-2\alpha}-o(n^{2-2\alpha}))-r^2\phi(r)\cdot \frac{n^{2-2\alpha}\eps^2}{2} \right). 
\end{align*}

By applying parts (a) and (c) of Lemma~\ref{le:totientProperties}, we obtain that
\begin{align}
I(\pts,\lines_3) &= \left(\frac{3}{\pi^2\eps^2}+o\left(\frac{1}{\eps^2}\right)\right)\cdot (n^{2-2\alpha}-o(n^{2-2\alpha}))-\left(\frac{3}{2\pi^2\eps^4}+o\left(\frac{1}{\eps^4}\right)\right)\cdot \frac{n^{2-2\alpha}\eps^2}{2} \nonumber \\[2mm]
&= \frac{3n^{2-2\alpha}}{\pi^2\eps^2} -\frac{3n^{2-2\alpha}}{4\pi^2\eps^2} + o\left(\frac{n^{2-2\alpha}}{\eps^2} \right) = \frac{9n^{2-2\alpha}}{4\pi^2\eps^2} + o\left(\frac{n^{2-2\alpha}}{\eps^2} \right). \label{eq:IncL325}
\end{align}

\parag{The rightmost columns with $\lines_4$.} 
We now study the lines of $\lines_4$, where the slopes are smaller than $-n^{1-2\alpha}$. 
Each such slope can be uniquely written as $-s/r$, where $r,s\in \NN$, $s>r\cdot n^{1-2\alpha}$, and $(r,s)=1$.
We fix values of $r$ and $s$ that satisfy these conditions.
The main difference from the analysis of $\lines_3$ is that now $n^{1-\alpha}/s < n^\alpha/r$.
Thus, in the current case \emph{maximal lines} contain $\lceil n^{1-\alpha}/s\rceil$ points of $\pts$, rather than $\lceil n^\alpha/r \rceil$.
As before, we may assume that $r\le 1/\eps$ and $s\le n^{1-2\alpha}/\eps$, since otherwise no line with slope $-s/r$ contains $\eps n^\alpha$ points of $\pts$.
When $\alpha=1/2$, we have that $|\lines_4|=|\lines_3|$ and that $I(\pts,\lines_4)=I(\pts,\lines_3)$.
It remains to consider the case where $1/3 \le \alpha <1/2$.

As before, let $\hat{\lines}_{s/r}$ be the set of $n^{1-\alpha}$ lines of slope $-s/r$ that are incident to a point from the rightmost column of $\pts$. 
Unlike the analysis of $\lines_3$, in this case we do not ignore the bottom $s$ points of the rightmost column.
Also as before, the top $s$ lines of $\hat{\lines}_{s/r}$ contain a single point of $\pts$.
The next $s$ top lines of $\hat{\lines}_{s/r}$ are incident to two points of $\pts$, the next $s$ lines to three points, and so on.
After getting to $s$ lines that are incident to $\lceil n^{1-\alpha}/s\rceil-1$ points of $\pts$, at most $s$ lines contain $\lceil n^{1-\alpha}/s\rceil$ points of $\pts$.

By removing from $\hat{\lines}_{s/r}$ lines that contain fewer than $\eps n^\alpha$ points of $\pts$, we obtain that
\begin{equation} \label{eq:LinesRightColumn25steep}
|\hat{\lines}_{s/r}\cap\lines_4| = n^{1-\alpha} - s(\eps n^\alpha-1) < n^{1-\alpha} - s \eps n^\alpha. 
\end{equation}
Imitating the analysis of the bottom row for $\lines_3$, we obtain that
\begin{align} 
I(\pts,\hat{\lines}_{s/r}\cap\lines_4) &=s\big(\eps n^\alpha+(\eps n^\alpha+1) +(\eps n^\alpha+2)+\cdots +n^{1-\alpha}/s- O(n^{1-\alpha}/s) \big)  \nonumber \\[2mm]
&=s\cdot \frac{(n^{1-\alpha}/s+\eps n^\alpha)(n^{1-\alpha}/s-\eps n^\alpha)}{2}- O(n^{1-\alpha}) \nonumber \\[2mm]
&= \frac{n^{2-2\alpha}}{2s}-\frac{s\eps^2n^{2\alpha}}{2} - O(n^{1-\alpha}). \label{eq:IncidencesRightColumn25steep}
\end{align}

Let $\hat{\lines}_{s/r}^*$ be the set of lines with slope $-s/r$ that are incident to a point in the $r$ rightmost columns of $\pts$ (including points from the bottom $s$ rows).  
By \eqref{eq:LinesRightColumn25steep}, we have that 
\begin{equation} \label{eq:LinesRightColumns25steep}
|\hat{\lines}_{s/r}^*\cap\lines_4| < r(n^{1-\alpha}- s\eps n^{\alpha}). 
\end{equation}
When moving from the rightmost column of $\pts$ to one of the other $r$ rightmost columns, the only difference is that some lines lose a single incidence. 
Being highly inefficient, the number of lost incidences is smaller than $rn^{1-\alpha}$.
In each of the columns, the number of lines that are no longer incident to at least $\eps n^\alpha$ decreases by at most $s$. 
In total, the number of incidences lost in this way is smaller than $rs\eps n^{\alpha}$.
Since $s\le n^{1-2\alpha}/\eps$, this amount is also smaller than $rn^{1-\alpha}$.
Combining the above with \eqref{eq:IncidencesRightColumn25steep} leads to
\begin{equation} \label{eq:IncidencesRightColumns25steep}
I(\pts,\hat{\lines}_{s/r}\cap\lines_4) = \frac{rn^{2-2\alpha}}{2s}-\frac{rs\eps^2n^{2\alpha}}{2} - O(rn^{1-\alpha}). 
\end{equation}

\parag{The bottom rows with $\lines_4$.} 
Let $\lines_{s/r}$ be the set of $n^\alpha-r$ lines of slope $-s/r$ that are incident to a point on the bottom row of $\pts$, excluding the $r$ rightmost points. 
The $r$ leftmost lines of $\lines_{s/r}$ contain a single point of $\pts$.
The next $r$ leftmost lines of $\lines_{s/r}$ are incident to two points of $\pts$. 
The next $r$ leftmost lines are incident to three points, and so on.
After $r$ lines that are incident to $\lceil n^{1-\alpha}/s\rceil-1$ points of $\pts$, the number of lines that contain $\lceil n^{1-\alpha}/s\rceil$ points (that is, of maximal lines) is 
\[ n^{\alpha} - r(\lceil n^{1-\alpha}/s\rceil-1) -r \ge n^{\alpha} - \frac{n^{1-\alpha}r}{s}. \]
Similarly, we have that
\begin{equation} \label{eq:LinesBottomRow25steep} 
|\lines_{s/r} \cap S_4| = n^\alpha - r(\eps n^\alpha-1) -r < n^\alpha - r\eps n^\alpha. 
\end{equation}

Adapting the analysis in \eqref{eq:IncidencesRightColumn25} leads to 
\begin{align} 
I(\pts,\lines_{s/r}\cap\lines_4) &\ge r\big(\eps n^\alpha+(\eps n^\alpha+1)+(\eps n^\alpha+2)+\cdots +(n^{1-\alpha}/s-1)\big) \nonumber \\
&\hspace{80mm}+ \frac{n^{1-\alpha}}{s} \cdot \left(n^{\alpha} - \frac{n^{1-\alpha}r}{s}\right)  \nonumber \\[2mm]
&= r\cdot\frac{(n^{1-\alpha}/s+\eps n^\alpha)(n^{1-\alpha}/s-\eps n^\alpha)}{2} + \frac{n}{s} - \frac{rn^{2-2\alpha}}{s^2} \nonumber \\[2mm]
&= \frac{rn^{2-2\alpha}}{2s^2}-\frac{r\eps^2n^{2\alpha}}{2} + \frac{n}{s} - \frac{rn^{2-2\alpha}}{s^2} = \frac{n}{s}-\frac{rn^{2-2\alpha}}{2s^2}-\frac{r\eps^2n^{2\alpha}}{2}.\label{eq:IncidencesBottomRow25steep}
\end{align}

Let $\lines_{s/r}^*$ be the set of lines with slope $-s/r$ that are incident to a point in the $s$ bottom rows of $\pts$, but not from the $r$ rightmost columns.  
A line with slope $-s/r$ contains a point of the integer lattice every $s$ rows, so such a line is incident to at most one point from the bottom $s$ rows.
Thus, $|\lines_{s/r}^*|=s (n^\alpha-r)$.
Equation \eqref{eq:LinesBottomRow25steep} implies that 
\begin{equation} \label{eq:LinesBottomRows25steep}
|\lines_{s/r}^*\cap\lines_4| \le sn^\alpha - rs\eps n^\alpha. 
\end{equation}

As in the previous cases, we can analyze the lines of $\lines_{s/r}^*$ in the same way we did the lines of $\lines_{s/r}$.
The only difference is that some of these lines now contain one point less. 
This decreases the number of incidences by less than $sn^\alpha$.
In each row, $r$ lines may now contain $\eps n^{\alpha}-1$ points of $\pts$, so not in $\lines_4$.
This decreases the number of incidences by less than $s r \eps n^\alpha$.
Since $r\le 1/\eps$, this quantity is again at most $s n^\alpha$.
Combining the above with \eqref{eq:IncidencesBottomRow25steep} implies that 
\begin{equation} \label{eq:IncidencesBottomRows25steep}
I(\pts,\lines_{s/r}^*\cap\lines_4) \ge n-\frac{rn^{2-2\alpha}}{2s}-\frac{rs\eps^2n^{2\alpha}}{2} - 2sn^\alpha. 
\end{equation}

\parag{Completing the analysis of $\lines_4$.} 
By the condition about the rightmost $r$ columns in the definition of $\lines_{s/r}^*$, we have that $\lines_{s/r}^*\cap \hat{\lines}_{s/r}^* = \varnothing$.
Every line of $\lines$ with slope $-s/r$ contains a point from the $s$ bottom rows or the $r$ rightmost columns of $\pts$.
In other words, every line of $\lines$ with slope $-s/r$ is either in $\lines_{s/r}^*$ or in $\hat{\lines}_{s/r}^*$.
By combining \eqref{eq:LinesRightColumns25steep} and \eqref{eq:LinesBottomRows25steep}, we get that the number of these lines is at most 
\begin{equation} \label{eq:TotalLinesFixedsr25steep}
r(n^{1-\alpha}- s\eps n^{\alpha}) + sn^\alpha - rs\eps n^\alpha = sn^\alpha + rn^{1-\alpha} -2rs\eps n^\alpha.
\end{equation}
By combining \eqref{eq:IncidencesRightColumns25steep} with \eqref{eq:IncidencesBottomRows25steep} and recalling that $r\le 1/\eps = o(n^\alpha)$, we get that the number of incidences with those lines is at least
\begin{align} 
\frac{rn^{2-2\alpha}}{2s}-\frac{rs\eps^2n^{2\alpha}}{2} - O(rn^{1-\alpha}) &+ n-\frac{rn^{2-2\alpha}}{2s}-\frac{rs\eps^2n^{2\alpha}}{2} - 2sn^\alpha \nonumber \\[2mm]
& = n - rs\eps^2n^{2\alpha} + o(n). \label{eq:TotalIncidencesFixedsr25steep}
\end{align}

We recall that $s>r\cdot n^{1-2\alpha}$ and that $s\le n^{1-2\alpha}/\eps$.
Combining this with \eqref{eq:TotalLinesFixedsr25steep}, Lemma~\ref{le:TotientPastN}, and Corollary~\ref{co:phiMN} implies that 
\begin{align*} 
|\lines_4| &\le \sum_{r=1}^{1/\eps} \sum_{r n^{1-2\alpha}< s \le n^{1-2\alpha}/\eps \atop (r,s)=1} \left(s\cdot(n^\alpha -2r\eps n^\alpha)+ rn^{1-\alpha} \right) \\
&= \sum_{r=1}^{1/\eps} \sum_{1\le s \le n^{1-2\alpha}/\eps \atop (r,s)=1} \left(s\cdot(n^\alpha -2r\eps n^\alpha)+ rn^{1-\alpha} \right) - \sum_{r=1}^{1/\eps} \sum_{1\le s \le r n^{1-2\alpha} \atop (r,s)=1} \hspace{-3mm} \left(s\cdot(n^\alpha -2r\eps n^\alpha)+ rn^{1-\alpha} \right) \\
&\le \sum_{r=1}^{1/\eps}  \bigg(\left\lceil \frac{n^{1-2\alpha}/\eps}{r}\right\rceil^2\cdot \frac{r\phi(r)}{2}\cdot n^\alpha - \left\lfloor \frac{n^{1-2\alpha}/\eps}{r}\right\rfloor^2\cdot \frac{r\phi(r)}{2} \cdot 2r\eps n^\alpha \\
&\hspace{70mm}+ \left(\frac{n^{1-2\alpha}}{r\eps}\cdot \phi(r)+O(2^{w(r)})\right)\cdot rn^{1-\alpha} \bigg) \\
&\hspace{40mm} - \sum_{r=1}^{1/\eps}  \left(n^{2-4\alpha}\cdot \frac{r\cdot \phi(r)}{2}\cdot (n^\alpha -2r\eps n^\alpha)+ n^{1-2\alpha}\cdot \phi(r) \cdot rn^{1-\alpha} \right). 
\end{align*}

To apply Lemma~\ref{le:TotientPastN}(c) with $1\le s\le n^{1-2\alpha}/\eps$, we require that $n^{1-2\alpha}/\eps$ is a multiple of $r$, which might not be the case.
For that reason, we round the positive term up and the negative term down, to the closest multiples of $r$.
This way, we maintain a valid upper bound. 
Since the ceiling function increases the expression in it by less than 1, we may replace the ceiling with an extra +1. 
Similarly, we may replace the floor function with an extra -1, obtaining that
\begin{align*} 
|\lines_4| &\le \sum_{r=1}^{1/\eps}  \bigg(\left( \frac{n^{1-2\alpha}}{r\eps}+1\right)^2\cdot \frac{r\phi(r)}{2}\cdot n^\alpha - \left( \frac{n^{1-2\alpha}}{r\eps}-1\right)^2\cdot r^2\phi(r) \cdot \eps n^\alpha + \frac{n^{2-3\alpha}}{\eps}\cdot \phi(r)\\
&\hspace{20mm}+O(r2^{w(r)}n^{1-\alpha}) - n^{2-3\alpha}\cdot \frac{r \phi(r)}{2} +n^{2-3\alpha}\cdot r^2 \phi(r)\cdot \eps - r\phi(r) \cdot n^{2-3\alpha}\bigg) \\
&=\sum_{r=1}^{1/\eps}  \bigg(\frac{n^{2-4\alpha}}{r^2\eps^2} \cdot \frac{r\phi(r)}{2}\cdot n^\alpha - \frac{n^{2-4\alpha}}{r^2\eps^2}\cdot r^2\phi(r) \cdot \eps n^\alpha + \frac{n^{2-3\alpha}}{\eps}\cdot \phi(r) - n^{2-3\alpha}\cdot \frac{3r \phi(r)}{2} \\
&\hspace{20mm}+n^{2-3\alpha}\cdot r^2 \phi(r)\cdot \eps+O\left(r2^{w(r)}n^{1-\alpha}+\frac{\phi(r)n^{1-\alpha}}{\eps}+r^2\phi(r)\eps n^{\alpha}\right)  \bigg) \\
&=\sum_{r=1}^{1/\eps}  \bigg(\frac{\phi(r)}{r}\cdot \frac{n^{2-3\alpha}}{2\eps^2} - r \phi(r)\cdot \frac{3n^{2-3\alpha}}{2} +r^2 \phi(r)\cdot \eps n^{2-3\alpha}\\
&\hspace{50mm}+O\left(r2^{w(r)}\cdot n^{1-\alpha}+\phi(r)\cdot \frac{n^{1-\alpha}}{\eps}+r^2\phi(r)\cdot \eps n^{\alpha}\right)  \bigg).
\end{align*}

We recall that $\alpha<1/2$, which implies that $n^{1-\alpha}=o(n^{2-3\alpha})$. 
By also relying on Lemma~\ref{le:totientProperties} and Lemma~\ref{le:sum2w(r)}, we have that
\begin{align}
|\lines_4| &\le \left(\frac{6}{\eps \pi^2} + o\left(\frac{1}{\eps}\right)\right)\cdot \frac{n^{2-3\alpha}}{2\eps^2} - \left(\frac{2}{\eps^3 \pi^2} + o\left(\frac{1}{\eps^3}\right)\right)\cdot \frac{3n^{2-3\alpha}}{2} \nonumber \\
&\hspace{10mm}+ \left(\frac{3}{2\eps^4 \pi^2} + o\left(\frac{1}{\eps^4}\right)\right)\cdot \eps n^{2-3\alpha}+O\left(\frac{\log\log (1/\eps)}{\eps^2}\cdot n^{1-\alpha}+\frac{1}{\eps^2}\cdot \frac{n^{1-\alpha}}{\eps}+\frac{1}{\eps^4}\cdot \eps n^{\alpha}\right) \nonumber \\[2mm]
&= \frac{3n^{2-3\alpha}}{\pi^2\eps^3} - \frac{3n^{2-3\alpha}}{\pi^2\eps^3} + \frac{3n^{2-3\alpha}}{2\pi^2\eps^3} + o\left(\frac{n^{2-3\alpha}}{\eps^3}\right) = \frac{3n^{2-3\alpha}}{2\pi^2\eps^3} + o\left(\frac{n^{2-3\alpha}}{\eps^3}\right). \label{eq:LinesL425}
\end{align}

Next, summing \eqref{eq:TotalIncidencesFixedsr25steep} over $r$ and $s$ and applying Corollary~\ref{co:phiMN} and Lemma~\ref{le:TotientPastN} leads to 
\begin{align*}
I(\pts,\lines_4) &\ge \sum_{r=1}^{1/\eps} \sum_{r n^{1-2\alpha}< s \le n^{1-2\alpha}/\eps \atop (r,s)=1} \left(n - rs\eps^2n^{2\alpha} + o(n)\right) \\
&= \sum_{r=1}^{1/\eps} \sum_{1\le s \le n^{1-2\alpha}/\eps \atop (r,s)=1} \left(n - rs\eps^2n^{2\alpha} + o(n)\right) - \sum_{r=1}^{1/\eps} \sum_{1\le s \le r n^{1-2\alpha} \atop (r,s)=1} \left(n - rs\eps^2n^{2\alpha} + o(n)\right) \\
&\ge \sum_{r=1}^{1/\eps}  \left((n+ o(n))\cdot \left(\frac{n^{1-2\alpha}/\eps}{r} \cdot \phi(r)+O(2^{w(r)})\right) - \left\lceil\frac{n^{1-2\alpha}/\eps}{r}\right\rceil^2\cdot \frac{r\phi(r)}{2}\cdot r\eps^2n^{2\alpha} \right) \\
&\hspace{30mm} -\sum_{r=1}^{1/\eps}  \left((n+ o(n))\cdot \phi(r) \cdot n^{1-2\alpha} - (n^{1-2\alpha})^2\cdot \frac{r\phi(r)}{2}\cdot r\eps^2n^{2\alpha} \right).
\end{align*}

To apply Lemma~\ref{le:TotientPastN}(c) with $1\le s\le n^{1-2\alpha}/\eps$, we require that $n^{1-2\alpha}/\eps$ is a multiple of $r$, which might not be the case.
By rounding the expression up to the closest multiple of $r$, we maintain a valid lower bound. 
Since the ceiling function increases the expression in it by less than 1, we may replace the ceiling with an extra +1. 
We obtain that 
\begin{align*}
I(\pts,\lines_4) &\ge \sum_{r=1}^{1/\eps}  \bigg(\frac{\phi(r)}{r}\cdot \left(\frac{n^{2-2\alpha}}{\eps}+o\left(\frac{n^{2-2\alpha}}{\eps}\right)\right) - \left(\frac{n^{1-2\alpha}}{r\eps}+1\right)^2\cdot \frac{r^2\phi(r)}{2}\cdot \eps^2n^{2\alpha}  \\
&\hspace{10mm} - \phi(r) \cdot (n^{2-2\alpha}+ o(n^{2-2\alpha})) + r^2\phi(r)\cdot \frac{\eps^2n^{2-2\alpha}}{2} + O(n2^{w(r)})\bigg) \\
&= \sum_{r=1}^{1/\eps}  \bigg(\frac{\phi(r)}{r}\cdot \left(\frac{n^{2-2\alpha}}{\eps}+o\left(\frac{n^{2-2\alpha}}{\eps}\right)\right) - \frac{n^{2-4\alpha}}{r^2\eps^2}\cdot \frac{r^2\phi(r)}{2}\cdot \eps^2n^{2\alpha} + r^2\phi(r)\cdot \frac{\eps^2n^{2-2\alpha}}{2} \\
&\hspace{10mm} - \phi(r) \cdot (n^{2-2\alpha}+ o(n^{2-2\alpha}))  + O(n2^{w(r)}+r\phi(r)\cdot \eps n+r^2\phi(r)\cdot \eps^2 n^{2\alpha})\bigg) \\
&= \sum_{r=1}^{1/\eps}  \bigg(\frac{\phi(r)}{r}\cdot \left(\frac{n^{2-2\alpha}}{\eps}+o\left(\frac{n^{2-2\alpha}}{\eps}\right)\right) + r^2\phi(r)\cdot \frac{\eps^2n^{2-2\alpha}}{2} \\
&\hspace{10mm} - \phi(r) \cdot \left(\frac{3n^{2-2\alpha}}{2}+ o(n^{2-2\alpha})\right)  + O(2^{w(r)}\cdot n+r\phi(r)\cdot \eps n+r^2\phi(r)\cdot \eps^2 n^{2\alpha})\bigg)
\end{align*}

We recall that $\alpha<1/2$, which implies that $n=o(n^{2-2\alpha})$ and $n^{2\alpha}=o(n^{2-2\alpha})$. 
By also relying on Lemma~\ref{le:totientProperties} and Lemma~\ref{le:sum2w(r)}, we have that
\begin{align}
I(\pts,\lines_4) &\ge \left(\frac{6}{\pi^2\eps}+o\left(\frac{1}{\eps}\right)\right)\cdot \left(\frac{n^{2-2\alpha}}{\eps}+o\left(\frac{n^{2-2\alpha}}{\eps}\right)\right) + \left(\frac{3}{2\pi^2\eps^4}+o\left(\frac{1}{\eps^4}\right)\right)\cdot \frac{\eps^2n^{2-2\alpha}}{2} \nonumber \\
&\hspace{30mm} - \left(\frac{3}{\pi^2\eps^2}+o\left(\frac{1}{\eps^2}\right)\right) \cdot \left(\frac{3n^{2-2\alpha}}{2}+ o(n^{2-2\alpha})\right) \nonumber \\
&\hspace{30mm} + O\left(\frac{n\log\log (1/\eps)}{\eps} +\frac{1}{\eps^3}\cdot \eps n+\frac{1}{\eps^4}\cdot \eps^2 n^{2\alpha}\right) \nonumber \\[2mm]
&=\frac{6n^{2-2\alpha}}{\pi^2\eps^2} + \frac{3n^{2-2\alpha}}{4\pi^2\eps^2} - \frac{9n^{2-2\alpha}}{2\pi^2\eps^2} + o\left(\frac{n^{2-2\alpha}}{\eps^2}\right) = \frac{9n^{2-2\alpha}}{4\pi^2\eps^2} + o\left(\frac{n^{2-2\alpha}}{\eps^2}\right). \label{eq:IncL425}
\end{align}

\parag{Wrapping up.}
We recall that $|\lines_1| = |\lines_3|$ and $|\lines_2|=|\lines_4|$.
When $\alpha=1/2$, we have the stronger observation $|\lines_1| = |\lines_2|=|\lines_3|=|\lines_4|$.
By combining this with \eqref{eq:LinesL325} and \eqref{eq:LinesL425}, we obtain that
\begin{equation}
|\lines| = |\lines_1| + |\lines_2| + |\lines_3| + |\lines_4| \le 4\cdot \frac{3n^{2-3\alpha}}{2\pi^2\eps^3} + o\left(\frac{n^{2-3\alpha}}{\eps^3}\right) = \frac{6n^{2-3\alpha}}{\pi^2\eps^3} + o\left(\frac{n^{2-3\alpha}}{\eps^3}\right). \label{eq:TotalLines}
\end{equation}

We recall that $I(\pts,\lines_1) = I(\pts,\lines_3)$ and $I(\pts,\lines_2) = I(\pts,\lines_4)$.
When $\alpha=1/2$, we have the stronger observation $I(\pts,\lines_1) = I(\pts,\lines_2) = I(\pts,\lines_3) = I(\pts,\lines_4)$.
By combining this with \eqref{eq:IncL325} and \eqref{eq:IncL425}, we obtain that
\begin{align}
I(\pts,\lines) = I(\pts,\lines_1) + I(\pts,\lines_2) + I(\pts,\lines_3) + I(\pts,\lines_4) \ge &4\cdot \frac{9n^{2-2\alpha}}{4\pi^2\eps^2} + o\left(\frac{n^{2-2\alpha}}{\eps^2}\right) \nonumber \\[2mm]
&= \frac{9n^{2-2\alpha}}{\pi^2\eps^2} + o\left(\frac{n^{2-2\alpha}}{\eps^2}\right). \label{eq:TotalInc}
\end{align}

Combining \eqref{eq:TotalLines} and \eqref{eq:TotalInc} leads to
\[ \frac{I(\pts,\lines)}{|\pts|^{2/3}|\lines|^{2/3}} = \frac{9n^{2-2\alpha}}{\pi^2\eps^2} \cdot \frac{(\pi^2\eps^3)^{2/3}}{n^{2/3}\cdot (6n^{2-3\alpha})^{2/3}} + o\left(1\right) = \frac{3^{4/3}}{\pi^{2/3}2^{2/3}} +o(1) \approx 1.27. \]

Recall that we may choose any $\eps$ that satisfies $\eps = o(1)$ and $\eps  = \omega(n^{-\alpha})$.
Thus, 
\[ |\lines|=\Theta\left(\frac{n^{2-3\alpha}}{\eps^3}\right) = o(n^2) = o(|\pts|^2) \]
and 
\[ |\lines|=\Theta\left(\frac{n^{2-3\alpha}}{\eps^3}\right) = \omega(n^{2-3\alpha}) = \omega(|\pts|^{2-3\alpha}). \]

\end{document}